\documentclass[10pt]{amsart}

\usepackage{amsthm,amssymb,amsmath,amsfonts,mathrsfs,amscd,yfonts,stmaryrd}
\usepackage{url,tikz,enumitem,caption,stmaryrd,comment,mathrsfs,marginnote,colonequals}
\usepackage[utf8]{inputenc}
\usepackage[T1]{fontenc}
\usepackage{libertine}
\usetikzlibrary{matrix,positioning,arrows,shapes,chains,calc,automata}
\usetikzlibrary{decorations.pathmorphing}
\usetikzlibrary{arrows,automata,positioning,knots}
\definecolor{darkgreen}{rgb}{0,0.5,0}
\usepackage[pagebackref=true,
        colorlinks, citecolor=darkgreen,
]{hyperref}

\newlength{\proofmargin}
\renewcommand*{\backref}[1]{}
\renewcommand*{\backrefalt}[4]{%
  \ifcase #1 %
    \relax
  \or
    $\uparrow$#2.%
  \else
    $\uparrow$#2.%
  \fi%
}

  \renewenvironment{thebibliography}[1]{%
    \begin{oldthebibliography}{#1}%
      \setlength{\parskip}{0ex}%
      \setlength{\itemsep}{0.1ex}%
      \setlength{\labelwidth}{1.4cm} 
  }%
  {%
    \end{oldthebibliography}%
  }

\DeclareMathOperator{\bbQ}{\mathbb{Q}}
\DeclareMathOperator{\bbF}{\mathbb{F}}
\DeclareMathOperator{\bbZ}{\mathbb{Z}}
\DeclareMathOperator{\bbG}{\mathbb{G}}

\DeclareMathOperator{\bbP}{\mathbb{P}}

\DeclareMathOperator{\cO}{\mathcal{O}}

\DeclareMathOperator{\fkp}{\mathfrak{p}}

\DeclareMathOperator{\Res}{Res}

\DeclareMathOperator{\Nm}{\mathrm{Nm}}

\newtheorem{theorem}{Theorem}

\newtheorem{corollary}[theorem]{Corollary}

\theoremstyle{remark}

\newtheorem{remark}[theorem]{Remark}

\makeatletter
\renewenvironment{proof}[1][\proofname]%
{%
\par\pushQED{\qed}\normalfont\topsep6\p@\@plus6\p@\relax%
\begin{list}{}{\rightmargin=8pt\leftmargin=\proofmargin}%
  \item[\hskip\labelsep\bfseries#1\@addpunct{.}]\ignorespaces
}{%
\popQED\end{list}\@endpefalse%
}%
\makeatother

\title[The unit equation has no solutions in number fields of degree prime to $3$ where $3$ splits completely]{The unit equation has no solutions in number fields of degree prime to $3$ where $3$ splits completely}

\author[Triantafillou]{Nicholas George Triantafillou}
\address{\hspace{-.2in}N. Triantafillou, Department of Mathematics, University of Georgia, Athens, GA 30602, USA}
\email{nicholas.triantafillou@uga.edu}

\begin{document}

\begin{abstract}
Let $K$ be a number field with ring of integers $\cO_{K}$. We prove that if $3$ does not divide $ [K:\bbQ]$ and $3$ splits completely in $K$, then the unit equation has no solutions in $K$. In other words,
there are no $x, y \in \cO_{K}^{\times}$ with $x + y = 1$.
Our elementary $p$-adic proof is inspired by the Skolem-Chabauty-Coleman method applied to the restriction of scalars of the projective line minus three points. 
Applying this result to a problem in arithmetic dynamics, we show that if $f \in \cO_{K}[x]$ has a finite cyclic orbit in $\cO_{K}$ of length $n$ then $n \in \{1, 2, 4\}$.
\end{abstract}

\maketitle



Let $K$ be a number field of degree $d$ over $\bbQ$ and let $\cO_{K}$ be the ring of integers of $K$. The set $E_{K} \colonequals \{x \in \cO_{K}^{\times}: 1-x \in \cO_{K}^{\times}\}$ of \emph{exceptional units} in $K$ is well-known to be finite, dating back to Siegel \cite{siegel-21}. Each $x \in E_{K}$ corresponds to a solution in $\cO_{K}^{\times}$ to the \emph{unit equation}, $x + y = 1$. Solutions to the unit equation and the $S$-unit equation (which allows $x$ and $1-x$ to be units up to a fixed-in-advance finite set $S$ of prime ideals) remain of substantial practical interest because of a wide variety of applications to number theory and other fields. These include: enumerating elliptic curves over $K$ with good reduction outside a fixed set of primes \cite{smart-97}; understanding finitely generated groups, arithmetic graphs, and recurrence sequences \cite{evertse1988s}; and many Diophantine problems \cite{gyory-92}.

Some work on exceptional units focuses on general upper bounds. Building on work of Baker and Gy\"ory on of linear forms in (complex/$p$-adic) logarithms, Evertse proved an explicit upper bound on $\# E_{K}$ which is exponential in the degree of $K$ \cite{evertse-84}. More recent work (e.g. \cite{gyory-19}) has refined these bounds somewhat, but the best known bounds remain exponential, while the `true' upper bound is conjectured by Stewart to be sub-exponential (see p. 120 of \cite{evertse1988s}.) See \cite{evertse-2015} for more applications and detail on upper bounds.

Other work focuses on low-degree number fields and/or computation. For instance, \cite{nagell-70} and \cite{niklasch-smart-98} study the number of exceptional units in fields of degree $3$ and $4$. There has also been recent progress \emph{computing} the set of solutions to $S$-unit equations over low-degree number fields \cite{akmrvw18} both in practice and as a test-case for computations by variants of Chabauty's method \cite{dan-cohen-wewers-15, triantafillou-19}.

Instead of studying low-degree $K$ or general upper bounds, we impose a local condition on $K$, showing:

 \begin{theorem} \label{thm:3-splits}
 	Let $K$ be a number field. Suppose that $3 \nmid [K:\bbQ]$ and $3$ splits completely in $K$. Then there is no solution to the unit equation in $K$. In other words, there is no pair $x,y \in \cO_{K}^\times$ such that $x + y = 1$. 
\end{theorem}
 
\begin{remark} 
The set of degree $d$ polynomials in $\bbZ[x]$ which generate number fields where $3$ splits completely have positive density (ordered by height). Indeed, if $g(x) = \sum_{i=0}^{d} a_{i} x^i$ satisfies $v_{3}(a_{d- i}) = i(i-1)/2$ for all $i$, a Newton polygon computation shows that the roots of $g$ have distinct $3$-adic valuations. If $g$ is also irreducible then $\bbQ[x]/(g(x))$ is a field where $3$ splits completely. The set of number fields $K$ where $3$ splits completely is  expected to have positive density in the set of degree $d$ number fields ordered by discriminant (for any $d$); there are precise conjectures of what this density should be \cite{bhargava-07}. 


Theorem~\ref{thm:3-splits} does not give the \emph{first}-known infinite family of number fields of high degree without exceptional units. Indeed, if \emph{any} prime $\fkp$ above $2$ in $K$ has residue field $\bbF_{\fkp} \cong \bbF_{2}$ then there are no exceptional units in $K$ for a trivial reason. The values $x$ and $1-x$ cannot simultaneously be non-zero modulo\,$\fkp$. To our knowledge, Theorem~\ref{thm:3-splits} yields the first-known infinite family of number fields of high degree without exceptional units outside of these trivial examples.
\end{remark}
 
\begin{remark}
The hypothesis that $3 \nmid [K:\bbQ]$ in Theorem~\ref{thm:3-splits} is necessary.
The set of degree $3$ number fields containing exceptional units has been well-understood since at least \cite{nagell-70}. One can construct infinitely many degree $3$ number fields with an exceptional unit and where $3$ splits completely as follows: 
 
Choose an integer $c \equiv 40 \pmod{81}$. Let $g(x) = (x + c) x (x-1) - 2x + 1$, which is irreducible over $\bbQ$ be the rational root theorem. Let $\alpha$ be a root of $g$.  Let $K = \bbQ(\alpha)$. Since $\Nm_{K/\bbQ}(\alpha) = -g(0) = -1$ and $\Nm_{K/\bbQ}(1-\alpha) = g(1) = -1$, we see that $\alpha$ is an exceptional unit. Since the minimal polynomial of $(\alpha -2)/3$ , namely $\frac{1}{27} g(3x + 2) = x^3 + \frac{c+5}{3} x^2 + \frac{c+2}{3} x + \frac{2c + 1}{27}$, has integer coefficients and is congruent to $x(x-1)(x+1)$ modulo $3$, we see that $3$ splits completely in $K$. 

\end{remark}

\begin{remark}
If we replace the hypotheses `$3 \nmid [K:\bbQ]$ and $3$ splits completely in $K$' with `$5 \nmid [K:\bbQ]$ and $5$ splits completely in $K$' then Theorem~\ref{thm:3-splits} becomes false. Let $g(x) = x^3 - 4x^2 + x + 1$, let $\alpha$ be any root of $g$, and let $K = \bbQ(\alpha)$. Then $5$ splits completely in $K$. Moreover, $\Nm_{K/\bbQ}(\alpha) = -g(0) = -1$ and $\Nm_{K/\bbQ}(1-\alpha) = g(1) = -1$, so $\alpha$ and $1-\alpha$ are both units, i.e. $\alpha$ is an exceptional unit.
\end{remark} 

\begin{proof}
    Suppose that $u, v \in \cO_{K}^{\times}$ satisfy $-u-v = 1$, so that $-u$ and $-v$ are solutions to the unit equation. Since $3$ splits completely in $K$, there are $d$ embeddings $\cO_{K} \hookrightarrow \bbZ_{3}$. Let $u_{1}, \dots, u_{d}$ be the images of $u$ in $\bbZ_{3}$ under these embeddings. Since $u$ and $v$ are units, $u_{i} \in 1 + 3 \bbZ_{3}$ for all $i \in \{1, \dots, d\}$. Also, $\Nm_{K/\bbQ}(u), \Nm_{K/\bbQ}(v) \in \bbZ^{\times} = \{\pm 1\}$. We have
    \begin{align*}
	\prod_{i=1}^{d} u_{i} = \Nm_{K/\bbQ}(u) = 1\, \qquad \text{and} \qquad
	\prod_{i=1}^{d}(1 + u_{i}) = \Nm_{K/\bbQ}(-v) = (-1)^{d}\,.
	\end{align*}
	
	We see that $n = 1$ is a zero of the $3$-adic analytic function
	\[
	f(n) \colonequals (1 + u_1^n) \cdots (1 + u_{d}^{n}) - (-1)^{d} \,
	\]
	and 
	\begin{align*}
	f(-n) & = \prod_{i=1}^{d} (1 + u_i^{-n}) - (-1)^{d}  = \prod_{i=1}^{d} u_{i}^{-n} \prod_{i=1}^{d} (1 + u_i^{n}) - (-1)^{d} = \prod_{i=1}^{d} (1 + u_i^{n}) - (-1)^{d} = f(n)\,.
	\end{align*}
	
	In particular, expanding $f$ as a $p$-adic power series, all coefficients in odd degrees are zero. Now,
	 \[
	f(n) = -(-1)^{d} + \prod_{i=1}^{d}(1 + \exp(n \log u_{i})),.
	\]
	Let $v_{3}$ be the $3$-adic valuation normalized so that $v_3(3) = 1$. Since $v_{3}(\log u_i) \geq 1$ and $\exp$ converges when $v_{3}(n \log u_i) > 1/2$ (see \cite{gouvea-97}), this expression converges for all $n \in \bbZ_{3}$. 
	
	Expanding $f$ as a power series,
	\[
	f(n) = 
	 -(-1)^d + \prod_{i=1}^{d}(2 + n \log u_i + \frac{n^2}{2} (\log u_i)^2 + \frac{n^3}{3!} (\log u_i)^3 + \cdots) \equalscolon \sum_{j =0}^{\infty} a_{j} n^j\,.
	\]
	Now,
	\begin{align*}
		a_{0}  = 2^d - (-1)^d\,, \quad a_{1}  = 0\,, \quad a_{2} = 2^{d-3} \sum_{i=1}^{d} (\log u_i)^2\,,  \quad a_{3} = 0\,, \quad \text{and} \quad v_{3}(a_{j}) \geq 3\,\, \text{for} \,\, j \geq 4\,.\\
	\end{align*}
	Since $v_3(a_{2}) \geq 2$ and $f(1) = 0$ we have $v_{3}(a_{0}) \geq 2$. But $v_{3}(2^d - (-1)^d) \geq 2$ if and only if $3|d$.
\end{proof}			

\begin{remark}
	The inspiration for the proof of Theorem~\ref{thm:3-splits} is a variant of the method of Skolem-Chabauty-Coleman applied to the \emph{restriction of scalars} of $\bbP^{1}_{\cO_{K}} \smallsetminus \{0,1,\infty\}$ from $\cO_{K}$ to $\bbZ$. In this setting, $\bbP^{1}_{\cO_{K}} \smallsetminus \{0,1,\infty\}$ embeds into its \emph{generalized} Jacobian $\bbG_{m, \cO_{K}} \times \bbG_{m, \cO_{K}}$ via the Abel-Jacobi map $x \mapsto (x, x-1)$. To prove that $\bbP^{1}\smallsetminus \{0,1,\infty\} = \emptyset$, we consider the restriction of scalars of the Abel-Jacobi map. In this language, the proof of Theorem~\ref{thm:3-splits} amounts to showing that for any unit $u \in \cO_{K}^{\times}$ the intersection
	\[
	E_{u} \colonequals (\Res_{\cO_{K}/\bbZ} \bbP^1\smallsetminus\{0,1,\infty\})(\bbZ_{3}) \cap \overline{\{u^n: n \in \bbZ\} \times \cO_{K}^{\times}}
	\]
	inside $(\Res_{\cO_{K}/\bbZ} (\bbG_m \times \bbG_{m}))(\bbZ_{3})$ is empty. Here, the closure on the right is respect to the $3$-adic topology. To conclude, $\bigcup_{u \in \cO_{K}^{\times}} E_{u} = \emptyset$ is the set of solutions to the unit equation in $K$. See \cite{triantafillou-19} for a more general discussion of using Skolem-Chabauty-Coleman applied to the \emph{restriction of scalars} to compute solutions to the $S$-unit equation.
\end{remark}

We share an application in arithmetic dynamics communicated to the author by W{\l}adys{\l}aw Narkiewicz. 

\begin{corollary}
	 	Let $K$ be a number field. Suppose that $3 \nmid [K:\bbQ]$ and $3$ splits completely in $K$. Suppose that $f \in \cO_{K}[x]$ has a finite orbit of size $n$ in $\cO_{K}$, (i.e., that there exist distinct $a_0, \dots, a_{n-1} \in \cO_{K}$ such that $f(a_{i}) = a_{i+1}$ for $i \in \{0,\ldots n-2\}$ and $f(a_{n-1}) = a_{0}$.) Then, $n \in \{1,2,4\}$.
\end{corollary}

\begin{proof}
	Since $\cO_{K}$ embeds in $\mathbb Z_{3}$, the $p=3$ case of Theorem~2 of \cite{pezda-94} says that $n \in \{1,2,3,4,6,9\}$. If $n$ is a multiple of $3$, replace $f$ with its $(n/3)$ iterate so that $f$ has finite orbit in $\cO_{K}$ of size exactly $3$. 
	
	Since $(a-b)|(f(a)-f(b)),$ it follows that $-\frac{a_1-a_2}{a_0 - a_1}, -\frac{a_2-a_0}{a_0 - a_1} \in \cO_{K}^{\times}$. These sum to $1$ and are therefore exceptional units. (This observation appears in \cite{narkiewicz-97}.) There are no exceptional units in $K$, so this is a contradiction, completing the proof.
	
	In fact, it is well-known (and elementary to prove) that there is a polynomial in $\cO_{K}[x]$ with a finite orbit of odd order in $\cO_{K}$ if any only if there is an exceptional unit in $K$. Using this fact, one can conclude that $n$ is a power of $2$ without using the result of Pezda.
\end{proof}

\subsection*{Acknowledgements}
Thank you to Joe Rabinoff and Bjorn Poonen for comments which  simplified the proof, to Pete Clark, K{\'a}lm{\'a}n Gy{\H{o}}ry, Dino Lorenzini, W{\l}adys{\l}aw Narkiewicz, and Paul Pollack, for helpful feedback on an early draft of this manuscript and to Vishal Arul, Jack Petok, and Padmavathi Srinivasan for helpful conversations. Thank you to the NSF Graduate Research Fellowship under grant \#1122374, Simons Foundation grant \#550033, and the RTG grant DMS-1344994 in Algebra, Algebraic Geometry, and Number Theory at UGA for funding this work.

\bibliographystyle{alpha}

\bibliography{biblio}

\newcommand{\etalchar}[1]{$^{#1}$}
\begin{thebibliography}{AKM{\etalchar{+}}18}

\bibitem[AKM{\etalchar{+}}18]{akmrvw18}
Alejandra Alvarado, Angelos Koutsianas, Beth Malmskog, Chris Rasmussen,
  Christelle Vincent, and Mckenzie West.
\newblock Solving {S}-unit equations over number fields.
\newblock \url{https://trac.sagemath.org/ticket/22148}, 2018.

\bibitem[Bha07]{bhargava-07}
Manjul Bhargava.
\newblock Mass formulae for extensions of local fields, and conjectures on the
  density of number field discriminants.
\newblock {\em International Mathematics Research Notices},
  2007(9):rnm052--rnm052, 2007.

\bibitem[DCW15]{dan-cohen-wewers-15}
Ishai Dan-Cohen and Stefan Wewers.
\newblock Explicit {C}habauty-{K}im theory for the thrice punctured line in
  depth 2.
\newblock {\em Proc. Lond. Math. Soc. (3)}, 110(1):133--171, 2015.

\bibitem[EG15]{evertse-2015}
Jan-Hendrik Evertse and K{\'a}lm{\'a}n Gy{\H{o}}ry.
\newblock {\em Unit equations in {D}iophantine number theory}, volume 146.
\newblock Cambridge University Press, 2015.

\bibitem[EGST88]{evertse1988s}
JH~Evertse, K~Gy{\H{o}}ry, CL~Stewart, and R~Tijdeman.
\newblock {$S$}-unit equations and their applications.
\newblock {\em New advances in transcendence theory}, pages 110--174, 1988.

\bibitem[Eve84]{evertse-84}
J.-H. Evertse.
\newblock On equations in {$S$}-units and the {T}hue-{M}ahler equation.
\newblock {\em Invent. Math.}, 75(3):561--584, 1984.

\bibitem[Gou97]{gouvea-97}
Fernando~Q Gouv{\^e}a.
\newblock p-adic numbers.
\newblock In {\em p-adic Numbers}, pages 43--85. Springer, 1997.

\bibitem[Gy{\H{o}}92]{gyory-92}
K{\'a}lm{\'a}n Gy{\H{o}}ry.
\newblock Some recent applications of {$S$}-unit equations.
\newblock {\em Ast{\'e}risque}, 209(11):17--38, 1992.

\bibitem[Gy{\H{o}}19]{gyory-19}
K{\'a}lm{\'a}n Gy{\H{o}}ry.
\newblock Bounds for the solutions of {$ S $}-unit equations and decomposable
  form equations ii.
\newblock {\em Publ. Math. Debrecen}, 94:507--526, 2019.

\bibitem[Nag70]{nagell-70}
Trygve Nagell.
\newblock Quelques probl{\`e}mes relatifs aux unit{\'e}s alg{\'e}briques.
\newblock {\em Arkiv f{\"o}r Matematik}, 8(2):115--127, 1970.

\bibitem[NP97]{narkiewicz-97}
W{\l}adys{\l}aw Narkiewicz and Tadeusz Pezda.
\newblock Finite polynomial orbits in finitely generated domains.
\newblock {\em Monatshefte f{\"u}r Mathematik}, 124(4):309--316, 1997.

\bibitem[NS98]{niklasch-smart-98}
Gerhard Niklasch and N~Smart.
\newblock Exceptional units in a family of quartic number fields.
\newblock {\em Mathematics of computation}, 67(222):759--772, 1998.

\bibitem[Pez94]{pezda-94}
Tadeusz Pezda.
\newblock Polynomial cycles in certain local domains.
\newblock {\em Acta Arithmetica}, 66(1):11--22, 1994.

\bibitem[Sie21]{siegel-21}
Carl Siegel.
\newblock Approximation algebraischer zahlen.
\newblock {\em Mathematische Zeitschrift}, 10(3-4):173--213, 1921.

\bibitem[Sma97]{smart-97}
Nigel~P Smart.
\newblock {$S$}-unit equations, binary forms and curves of genus 2.
\newblock {\em Proceedings of the London Mathematical Society}, 75(2):271--307,
  1997.

\bibitem[Tri19]{triantafillou-19}
Nicholas Triantafillou.
\newblock {\em Restriction of Scalars, the Chabauty–Coleman Method, and
  $\mathbb P^1 \smallsetminus \{0, 1, \infty\}$}.
\newblock PhD thesis, Massachusetts Institute of Technology, 6 2019.
\newblock Publicly available at
  \url{https://ngtriant.github.io/papers/Triantafillou-Thesis_Final.pdf}.

\end{thebibliography}


\end{document}